\documentclass[10pt]{amsart}
\usepackage[margin=1.2in,marginparsep=0.1in,marginparwidth=1in]{geometry}
\usepackage{amssymb,amsmath,amsthm,amstext,amscd,latexsym,graphics,graphicx,bbm,caption}
\usepackage[dvipsnames,svgnames,table]{xcolor}
\usepackage[plainpages=false,colorlinks=true]{hyperref}
\usepackage{tikz}
\usepackage{tikz-cd}
\usepackage{booktabs}
\usetikzlibrary{positioning}
\usepackage{mathtools,cleveref}
\usepackage{comment}
\tikzset{>=stealth}
\hypersetup{citecolor=Sepia,linkcolor=blue, urlcolor=blue}

\usepackage{cite} 

\makeatletter
\def\@tocline#1#2#3#4#5#6#7{\relax
  \ifnum #1>\c@tocdepth 
  \else
    \par \addpenalty\@secpenalty\addvspace{#2}%
    \begingroup \hyphenpenalty\@M
    \@ifempty{#4}{%
      \@tempdima\csname r@tocindent\number#1\endcsname\relax
    }{%
      \@tempdima#4\relax
    }%
    \parindent\z@ \leftskip#3\relax \advance\leftskip\@tempdima\relax
    \rightskip\@pnumwidth plus4em \parfillskip-\@pnumwidth
    #5\leavevmode\hskip-\@tempdima
      \ifcase #1
       \or\or \hskip 2em \or \hskip 2em \else \hskip 3em \fi%
      #6\nobreak\relax
    \dotfill\hbox to\@pnumwidth{\@tocpagenum{#7}}\par
    \nobreak
    \endgroup
  \fi}
\makeatother
\definecolor{deepgreen}{RGB}{0,100,0}
\newtheorem{intro-thm}{Theorem}[]
\theoremstyle{plain}
\newtheorem{thm}{Theorem}[section]
\newtheorem{theorem}[thm]{Theorem}
\newtheorem{question}[thm]{Question}
\newtheorem{lemma}[thm]{Lemma}

\theoremstyle{definition}

\newtheorem{definition}[thm]{Definition}

\newcommand{\GL}{{\mathrm {GL}}}

\newcommand{\sA}{{\mathcal A}}
\newcommand{\sB}{{\mathcal B}}
\newcommand{\sC}{{\mathcal C}}
\newcommand{\sD}{{\mathcal D}}

\newcommand{\sF}{{\mathcal F}}

\newcommand{\sH}{{\mathcal H}}
\newcommand{\sI}{{\mathcal I}}

\newcommand{\sK}{{\mathcal K}}

\newcommand{\sM}{{\mathcal M}}

\newcommand{\sT}{{\mathcal T}}

\newcommand{\C}{{\mathbb C}}
\newcommand{\D}{{\mathbb D}}

\newcommand{\R}{{\mathbb R}}
\newcommand{\T}{{\mathbb T}}

\newcommand{\Z}{{\mathbb Z}}

\newcommand\restr[2]{\ensuremath{\left.#1\right|_{#2}}}

\input{xy}
\xyoption{all}

\DeclareMathOperator{\ind}{ind}
\DeclareMathOperator{\ran}{ran}

\begin{document}

\author{Sourav Ghosh}
\address{Department of Mathematics, Indian Institute of Science Education and Research Pune, Maharashtra 411008}
\email{sourav.ghosh@students.iiserpune.ac.in}
\author{Haripada Sau}
\address{Department of Mathematics, Indian Institute of Science Education and Research Pune, Maharashtra 411008}
\email{hsau@iiserpune.ac.in}
\thanks{The author Sau was supported by the Anusandhan National Research Foundation (ANRF), Government of India, under the Mathematical Research Impact-Centric Support (MATRICS) scheme (Grant No. ANRF/ARGM/2025/000573/MTR)}

\thanks{}

\title[Unitary invariant]{A unitary invariant for $C^*$-algebras generated by isometries with a twisted commutation relation}

\vspace{2mm}


    
\begin{abstract}
We investigate the classification of $C^*$-algebras generated by pairs of isometries $(V_1, V_2)$ satisfying $V_1V_2=qV_2V_1$, where $q=e^{2\pi i\theta}$ with $\theta$ irrational, up to unitary equivalence. Using a functional model that realizes pure $q$-commuting pairs on vector-valued Hardy spaces in terms of a uniquely associated triple $(\mathcal{F}, P, U)$ consisting of Hilbert space $\mathcal{F}$, an orthogonal projection $P$ and a unitary operator $U$, we analyze the internal structure of $C^*(V_1,V_2)$ via an extension of the $C^*$-algebra generated by its minimal $q$-commuting unitary extension by the defect ideal $\mathcal{I}_{\mathcal{D}} =\langle [V_1^*, V_1],[V_2^*, V_2]\rangle$. For pairs that are essentially doubly $q$-commuting, we show that $PUP|_{\operatorname{ran} P}$ is semi-Fredholm and prove that $|\operatorname{ind} PUP|_{\operatorname{ran} P}|$ is invariant under unitary equivalence of the corresponding $C^*$-algebras, thereby providing an explicit criterion for distinguishing non-equivalent pairs.

This work extends the seminal 1978 results of Berger, Coburn, and Lebow for commuting isometries, while addressing the key technical nuances of this noncommutative setting: most notably, navigating the failure of Gelfand theory due to noncommutativity and replacing the classical Toeplitz extension by an extension of the noncommutative torus by the ideal of compact operators.

\end{abstract}

\maketitle

\section{Introduction}
We study unitary equivalence of $C^*$-algebras generated by pairs of isometries on a Hilbert space satisfying a twisted commutation relation. Fix $q=e^{2\pi i\theta}$ for some $\theta\in \R$, and let $(V_1,V_2)$ be a pair of isometries on a Hilbert space $\sH$. We call $(V_1,V_2)$ \emph{$q$-commuting} if
\begin{equation}\label{E:qComm}
V_1V_2=qV_2V_1
\end{equation}holds. A prototypical example of a $q$-commuting isometric pair is $(R_q,T_z)$ on the classical Hardy space $H^2$, where $R_q$ is the rotation operator and $T_z$ is the shift operator, defined respectively by
$$
R_q:f(z)\mapsto f(qz) \quad\mbox{and}\quad
T_z:f(z)\mapsto zf(z).
$$ By the functional model theory for $q$-commuting isometries developed in \cite{MR5075249} (briefly discussed in \Cref{S:q-model}), the operators $(R_q,T_z)$ acting on Hardy spaces (possibly with non-trivial coefficient spaces) serve as building blocks for $q$-commuting isometries. The universal $C^*$-algebra generated by $q$-commuting isometric symbols has been studied extensively; see \cite{MR2132080} for classification up to isometric $*$-isomorphism and \cite{MR3022732} for representations, nuclearity and $K$-theory. The concrete algebra $C^*(R_q,T_z)$ was studied by Park \cite{MR2176158} via an extension of the irrational rotation algebra; in particular, an index invariant is defined using differential geometric techniques. Jury later generalized this framework in \cite{MR2320853} to the $C^*$-algebra generated by $T_z$ and the composition operator $C_{\varphi}:H^2\to H^2$ defined by $C_{\varphi}(f)=f\circ \varphi$, for a linear fractional self map $\varphi$ of the unit disk $\mathbb D$.

Our focus here is unitary equivalence of such concrete $C^*$-algebras. Let $\sA_1\subset\sB(\sH_1)$ and $\sA_2\subset\sB(\sH_2)$ be $C^*$-algebras represented on Hilbert spaces $\sH_1$ and $\sH_2$, respectively. We say that $\sA_1$ and $\sA_2$ are \emph{unitarily equivalent} if there exists a unitary operator $\tau:\sH_1\to \sH_2$ such that
$$
\sA_2=\tau \sA_1\tau^*\coloneqq \{\tau T\tau^*: T\in \sA_1\}.
$$
In this case $\sA_1$ and $\sA_2$ are isometrically $*$-isomorphic via the conjugation map $\Phi:\sA_1\to \sA_2$, defined by $\Phi(T)=\tau T\tau^*$.

Motivated by classification problems, we ask when two $q$-commuting isometric pairs generate unitarily equivalent $C^*$-algebras.
\begin{question}\label{Q:TheQuestion}
Given two $q$-commuting isometric pairs $(V_1,V_2)$ and $(V_1',V_2')$, when are the associated $C^*$-algebras $C^*(V_1,V_2)$ and $C^*(V_1',V_2')$ unitarily equivalent?
\end{question}
We now summarize our main results. Our first main result (\Cref{T:kappa_Cstar_invariant}) provides a necessary condition for unitary equivalence under the hypotheses that $\theta$ is irrational and $(V_1,V_2)$ is \emph{essentially doubly $q$-commuting}, meaning that   $V_2V_1^*-qV_1^*V_2$ is compact.

We say that an isometric pair $(V_1,V_2)$ is \emph{pure} if
\begin{align}\label{pure}
\bigcap_{n=0}^{\infty}(V_1V_2)^n\sH=\{0\}.
\end{align}Our main tool is a functional model developed in \cite{MR5075249}, which realizes a $q$-commuting \emph{pure} isometric pair $(V_1,V_2)$ as operators on a vector-valued Hardy space $H^2(\sF)=H^2\otimes\sF$ in the form
\begin{align}\label{BS-model}
\Bigl(
T_z R_q \otimes UP + R_q\otimes UP^\perp,
\quad
R_{\overline{q}}T_z \otimes P^\perp U^* + R_{\overline{q}}\otimes PU^*
\Bigr).
\end{align} 
where the triple $(\sF,P,U)$, herein called a \textit{BCL triple}, consists of a Hilbert space $\sF$, an orthogonal  projection $P$ and a unitary operator $U$, all uniquely associated with $(V_1,V_2)$ -- see Theorem \ref{T:q-BCL} for a more precise statement. The model \eqref{BS-model} extends the classical work of Berger, Coburn and Lebow \cite{MR467392}, which is the case $q=1$. In the same paper they provide a necessary condition for unitary equivalence (see Question \ref{Q:TheQuestion}) in the commutative setting which also motivated the present investigation. 

The functional model \eqref{BS-model} yields a concrete description of $C^*(V_1,V_2)$ and clarifies its relationship with the $C^*$-algebra generated by the minimal $q$-commuting unitary extension of $(V_1,V_2)$. In particular, we show in \Cref{T:Represent} that every $X\in C^*(V_1,V_2)$ can be written as
\begin{align}\label{Represent}
X=T_A+D,
\end{align}
where $A$ belongs to the $C^*$-algebra of the minimal $q$-commuting unitary extension and $D$ lies in the \emph{defect ideal}
\begin{align}\label{Defect}
\sI_{\sD}\coloneqq \left\langle [V_1^*,V_1],[V_2^*,V_2]\right\rangle.
\end{align}

Let $(V_1,V_2)$ be a $q$-commuting isometric pair with associated \emph{BCL triple} $(\sF,P,U)$. When the Calkin algebra $C^*(P,U,\sK)/\sK$ is commutative, we obtain an index theory for Fredholm operators of the form $I+D$ with $D\in\sI_{\sD}$. We apply this to the distinguished element of the defect ideal,
\begin{align*}
D(V_1,V_2)\coloneqq [V_2^*,V_2](V_1-I)[V_2^*,V_2].
\end{align*}
These ingredients lead to a unitary invariant in a more general setting. When $(V_1,V_2)$ is essentially doubly $q$-commuting, we show that $PUP\vert_{\mathrm{ran} P}$ is semi-Fredholm and 
\begin{equation}\label{E:kappa}
\kappa(V_1,V_2)\coloneqq \ind \bigl(PUP\vert_{\mathrm{ran}P}\bigr)\in \Z\cup\{+\infty\}.
\end{equation}
We prove that 
$|\kappa(V_1,V_2)|$ is a unitary invariant, i.e., if $C^*(V_1,V_2)$ and $ C^*(V_1',V_2')$ are unitarily equivalent, then $
|\kappa(V_1,V_2)|=|\kappa(V_1',V_2')|.
$

Although our approach follows \cite{MR467392}, the twisted setting introduces additional technical hurdles. The main obstruction is that, unlike the case $q=1$, the $C^*$-algebra $C^*(\widetilde{V}_1,\widetilde{V}_2)$ generated by a $q$-commuting unitary extension $(\widetilde{V}_1,\widetilde{V}_2)$ of $(V_1,V_2)$ is noncommutative. Consequently, the commutator ideal
\[
\mathcal{C}=\overline{\mathrm{span}}\{AB-BA: A,B\in C^*(V_1,V_2)\}
\]
no longer plays the role it does in the commuting case. We show instead that the defect ideal $\mathcal{I}_D$ in \eqref{Defect} is the appropriate replacement that leads to the decomposition \eqref{Represent}.
The second difficulty is that noncommutativity makes Gelfand theory unavailable. When $C^*(\widetilde{V}_1,\widetilde{V}_2)$ is commutative, Gelfand theory provides a powerful framework for studying the Fredholm index of operators of the form $T_A+D$ with $A\in C^*(\widetilde{V}_1,\widetilde{V}_2)$. In the twisted case this approach breaks down. Nevertheless, in \Cref{S:Unitary invariant} we show that it suffices to analyze the special case $A=I$ in order to construct the unitary invariant \eqref{E:kappa}. We rely on the rich structure of noncommutative torus to mitigate the absence of Gelfand theory.

\section{Preliminaries}
In this section, we fix notation and briefly recall the Fredholm theory and operator models used throughout the paper.

\subsection{Notation}
We denote by $\sB(\sH)$ and $\sK(\sH)$ the $C^*$-algebra of bounded operators and the ideal of compact operators on a Hilbert space $\sH$, respectively. When it is clear from the context we use the convention $\sK(\sH)=\sK$.

Throughout, $\D\subset\C$ denotes the open unit disk and $\T$ the unit circle. We write $H^2$ for the Hardy space on $\D$, realized as a closed subspace of $L^2(\T)$ via non-tangential boundary values. Let $(e_n)_{n\geq 0}$ be the standard orthonormal basis of $H^2$. For a Hilbert space $\sF$, we identify
$
H^2(\sF)\cong H^2\otimes \sF,
$
which we view as the space of $\sF$-valued analytic functions on $\D$ with square-summable Taylor coefficients about the origin. We denote by $P_{H^2}$ the orthogonal projection of $L^2(\T)$ onto $H^2$.

We use the standard notation $M_z\in\sB(L^2(\T))$ for the bilateral shift: $(M_z f)(z)=zf(z)$. With a slight abuse of notation, we denote $R_q$ for the rotation operator $R_q: f\mapsto f(qz)$ for $f$ in $L^2(\T)$ as well. For an orthogonal projection $P$ we write $P^\perp\coloneqq I-P$. For $n\geq 0$, $\widetilde{P_n}$ denote the rank-one projection onto $\textrm{span} \{e_n\}$.

Given a collection $\mathcal S$ of operators on $\sH$, we use $C^*(\mathcal S)$ and
$
\langle \mathcal S \rangle
$ to denote, respectively, the $C^*$-algebra and the ideal (closed two-sided $*$-ideal) generated by $\mathcal S$. We also write $[A,B]$ for the commutator $AB-BA$ of two operators $A$ and $B$.

With these notations in place, we now briefly recall two themes in operator theory whose interaction constitutes the principal subject of this paper.

\subsection{Model of \texorpdfstring{$q$}{q}-commuting isometries}\label{S:q-model}
The functional model theory for commuting pairs of isometries traces its foundations to the seminal works \cite{MR467392, MR2931928, MR2289758, MR2091468, MR229093}; see also \cite{MR4978050} and \cite[Chapter 3]{MR5090861} for a comprehensive presentation and further developments. Extending the commuting framework, the work \cite{MR5075249} establishes a parallel functional model for isometric pairs satisfying a twisted commutation relation. since this result plays a central role in our investigation, we record it here for ease of reference.
\begin{theorem}[See Theorems 3.6 and 3.10 in \cite{MR5075249}] \label{T:q-BCL}
If $(V_1,V_2)$ is a pure pair (see \eqref{pure} for the definition) of $q$-commuting isometries on $\sH$, then there exist a Hilbert space $\sF$, a projection $P\in\sB(\sF)$, and a unitary $U\in\sB(\sF)$ such that $(V_1,V_2)$ is unitarily equivalent to
\begin{align}\label{isomod1}
\Bigl(
T_z R_q \otimes UP + R_q\otimes UP^\perp,
\quad
R_{\overline{q}}T_z \otimes P^\perp U^* + R_{\overline{q}}\otimes PU^*
\Bigr)
\quad \text{on} \quad H^2\otimes \sF.
\end{align}

The triple $(\sF,P,U)$ is uniquely determined up to unitary equivalence, i.e., if $(V_1,V_2)$ and $(V_1',V_2')$ are pairs of $q$-commuting isometries on $\sH$ and $\sH'$, with associated triple $(\sF,P,U)$ and $(\sF',P',U')$, then $(V_1,V_2)$ and $(V_1',V_2')$ are jointly unitarily equivalent if and only if there exists a unitary  $\tau:\sF\to \sF'$ such that $\tau(P,U)=(P',U')\tau$.
\end{theorem}
It is shown in \cite[Theorem 3.6]{MR5075249} that a canonical choice for the triple $(\sF,P,U)$ is 
$$
(\sF,P,U
)=\left( 
\begin{bmatrix}
    \sD_{V_1^*}\\
    \sD_{V_2^*}
\end{bmatrix},
\begin{bmatrix}
    0 & 0 \\ 0 & I_{\sD_{V_2^*}}
\end{bmatrix},
U:
\begin{bmatrix}
  D_{V_1^*}V_2^* \\
  D_{V_2^*}
\end{bmatrix}h\mapsto\begin{bmatrix}
  D_{V_1^*} \\
  D_{V_2^*}V_1^*
\end{bmatrix}h
\right).
$$The unitary $U$ above may seem to be only partially defined but by \cite[Lemma 3.5]{MR5075249}, it is uniquely determined. The unitary operator identifying the abstract pair $(V_1,V_2)$ with the model pair \eqref{isomod1} is
\begin{align}
\tau_{\rm BCL} : h \mapsto \begin{bmatrix}
    D_{V_1^*} \\
    D_{V_2^*}V_1^*
\end{bmatrix}(I_\sH - zV_2^*V_1^*)^{-1}h.
\end{align}For details, see the proof of Theorem 3.6 in \cite{MR5075249}. The unitary identification $\tau_{\rm BCL}$ and the triple $(\sF,P,U)$ do not depend on the unimodular number $q$. In view of this canonicity, we call the triple $(\sF,P,U)$ obtained in Theorem \ref{T:q-BCL} the \emph{BCL triple} associated with $(V_1,V_2)$. Although, we do not need the explicit form of $\tau_{\rm BCL}$, we retain this notation for later use.

It is convenient to work with a functional model rather than abstract operators. In particular, the explicit model shows that the commuting isometric pair
$$
\Bigl(
T_z \otimes UP + I_{H^2}\otimes UP^\perp,
\quad
T_z \otimes P^\perp U^* + I_{H^2}\otimes PU^*
\Bigr)
$$
is essentially doubly commuting if and only if the associated $q$-commuting isometric pair
$$
\Bigl(
T_z R_q \otimes UP + R_q\otimes UP^\perp,
\quad
R_{\overline{q}}T_z \otimes P^\perp U^* + R_{\overline{q}}\otimes PU^*
\Bigr)
$$is essentially doubly commuting. This is further equivalent to 
$PUP^\perp$ is compact. A detailed computation of these assertions can be found in \cite[Theorem 4.1]{MR5075249}. We will use this passage between commuting isometric pairs, $q$-commuting isometric pairs, and projection/unitary data throughout the paper; see \cite[Section 2]{MR5075249} for a coordinate-free discussion of the same.

Theorem \ref{T:q-BCL} also yields an explicit minimal $q$-commuting unitary extension $(\widetilde{V}_1,\widetilde{V}_2)$ acting on $L^2(\T)\otimes\sF$ of a $q$-commuting isometric pair, viz.,
\[
(\widetilde{V}_1,\widetilde{V}_2)\coloneqq
\Bigl(
M_z R_q \otimes UP + R_q\otimes UP^\perp,
\quad
R_{\overline{q}}M_z \otimes P^\perp U^* + R_{\overline{q}}\otimes PU^*
\Bigr).
\]This minimal unitary extension is rather strong in the sense that in fact, the product unitary $\widetilde{V}_1\widetilde{V}_2=M_z$ is a minimal unitary extension of the product isometry $V_1V_2$.

\subsection{Fredholm theory}
Fredholm theory for bounded operators on a Hilbert space will be used repeatedly, in particular the characterization of Fredholm and semi-Fredholm operators via the Calkin algebra and the resulting index. The reader is referred to the book \cite{MR1634900} for the basics of Fredholm theory. We briefly recall the definitions and basic facts needed later.

A bounded operator $T$ on $\sH$ is \emph{Fredholm} if $\ran(T)$ is closed and both $\ker T$ and $\ker T^*$ are finite-dimensional. It is \emph{semi-Fredholm} if $\ran(T)$ is closed and at least one of $\ker T$ or $\ker T^*$ is finite-dimensional.

Let $\pi:\sB(\sH)\to \sB(\sH)/\sK$ be the canonical quotient map onto the Calkin algebra. By Atkinson's theorem (see e.g., \cite[Theorem 5.17]{MR1634900}), $T$ is Fredholm if and only if $\pi(T)$ is invertible. Moreover, \cite[Theorem~2.1]{MR467392} gives an algebraic characterization of semi-Fredholmness: $T$ is left (respectively, right) semi-Fredholm if and only if $\pi(T)$ is left (respectively, right) invertible. We refer to $\pi(T)$ as the \emph{symbol} of $T$. For a semi-Fredholm operator $T$, the \textit{index} is
$$
\ind (T)=\dim\ker T-\dim\ker T^*.
$$

We will also use a Fredholm criterion in tensor-product algebras, which we recall from \cite{MR467392}. 
\begin{theorem}[See Theorem~2.2 in \cite{MR467392}]\label{T:CPcross}
Let $\sA$ be a $C^*$-algebra and $\sI\subseteq \sA$ a closed two-sided $*$-ideal. If the quotient map $\pi:\sA\to \sA/\sI$ admits a completely positive cross section $\rho$, then for any $C^*$-algebra $\sB$, $\sI\otimes \sB$ is a closed two-sided ideal in $\sA\otimes \sB$ and the quotient algebra $(\sA\otimes \sB)/( \sI\otimes \sB)$ is $*$-isomorphic to $(\sA/\sI)\otimes \sB$.
\end{theorem}

If $\sA/\sI$ is commutative and separable, then the quotient map $\pi:\sA\to \sA/\sI$ admits a completely positive cross section (see e.g., \cite{MR69403,MR355560}). We will apply this observation to the coefficient algebra $C^*(P,U,\sK)/\sK$ to guarantee the existence of a completely positive cross section to invoke the following theorem.

\begin{theorem}[See Theorem~2.3 in \cite{MR467392}]\label{T:quotient} Let $\sA$ and $\sB$ be unital $C^*$-algebras faithfully represented on $\sH$ and $\widetilde{\sH}$, respectively, with $\sK\subseteq \sA$ and $\sK\subseteq \sB$ where $\sK$ always denotes the ideal of compact operators on the relevant Hilbert space. Let $\pi:\sA\to \sA/\sK$ and $\pi':\sB\to \sB/\sK$ be the corresponding quotient maps, and assume that both $\pi$, $\pi'$ admit completely positive cross sections. Then $(\sA\otimes \sB)/\sK$ is $*$-isomorphic to a closed $*$-subalgebra of $((\sA/\sK)\otimes \sB)\oplus (\sA\otimes (\sB/\sK))$.
\end{theorem}

As a consequence, an element $T\in \sA\otimes \sB$ is Fredholm if and only if the pair
\begin{align}\label{Eq:symbol}
\bigl((\pi\otimes 1)T,(1\otimes \pi')T\bigr)
\end{align}
is invertible in $((\sA/\sK)\otimes \sB)\oplus (\sA\otimes (\sB/\sK))$. We regard \eqref{Eq:symbol} as the \emph{symbol} of $T$ and refer to $((\sA/\sK)\otimes \sB)\oplus (\sA\otimes (\sB/\sK))$ as its symbol space. 

We now prove a preliminary lemma that will be used later.
\begin{lemma}\label{L:block-unitary-index}
Let $W\in\sB(\sH)$ be unitary and let $\sM\subset \sH$ be a closed subspace. With respect to the orthogonal decomposition $\sH=\sM\oplus\sM^{\perp}$, write
$$
W=\begin{bmatrix}A&B\\ C&D\end{bmatrix}.
$$If $B\colon\sM^{\perp}\to\sM$ is compact, then $A$ and $D$ are semi-Fredholm and
$$
\ind A=-\ind D.
$$
Moreover, $A$ is Fredholm if and only if $C$ is also compact.
\end{lemma}

\begin{proof}
From the $(1,1)$-entry of $WW^*=I$ we obtain $AA^*+BB^*=I_{\sM}$. Since $BB^*$ is compact, the image of $A$ in the Calkin algebra $\sB(\sM)/\sK$ is right-invertible; hence $A$ is (right) semi-Fredholm. Similarly, the $(2,2)$-entry of $W^*W=I$ gives $B^*B+D^*D=I_{\sM^{\perp}}$, so $D$ is semi-Fredholm. We claim that the restriction $$
C\mid_{\ker A}\colon \ker A\to \ker D^*
$$ is unitary. 

To see that it is an isometry, let $x\in\ker A$. From the $(2,1)$-entry of $W^*W=I$ we have $B^*Ax+D^*Cx=0$, so $Cx\in\ker D^*$. From the $(1,1)$-entry of $W^*W=I$ we have $A^*Ax+C^*Cx=x$, hence $C^*Cx=x$.

For surjectivity, let $y\in\ker D^*$. The $(1,2)$-entry of $WW^*=I$ yields $AC^*y+BD^*y=0$, so $C^*y\in\ker A$. Using the $(2,2)$-entry of $WW^*=I$, namely $CC^*+DD^*=I_{\sM^{\perp}}$, and $D^*y=0$, we obtain $CC^*y=y$.

Therefore $\dim\ker A=\dim\ker D^*$. A similar computation shows that the map $B^*\mid_{\ker A^*}\colon\ker A^*\to\ker D$ is unitary, so $\dim\ker A^*=\dim\ker D$. Consequently,
$$
\ind A=\dim\ker A-\dim\ker A^*=\dim\ker D^*-\dim\ker D=-\ind D.
$$
For the final assertion, first assume that $C$ is compact. Then the $(1,1)$-entry of $W^*W$ gives $A^*A+C^*C=I_{\sM}$, so $A^*A-I_{\sM}\in\sK$ and hence $A$ is left semi-Fredholm. Together with the right semi-Fredholmness, this implies that $A$ is Fredholm. 

Conversely, assume that $A$ is Fredholm. Since $AA^*-I\in\sK$, we have $\pi(A)^{-1}=\pi(A^*)$ in $\sB(\sM)/\sK$. Therefore, $\pi(I-A^*A)=0$. Hence $C^*C=I-A^*A\in\sK$, which implies that $C$ is compact.
\end{proof}

\section{Tensor representation}
The operator model in Theorem~\ref{T:q-BCL} for a $q$-commuting isometric pair $(V_1,V_2)$ yields a unitary identification of $C^*(V_1,V_2)$ with a concrete $C^*$-subalgebra of $\sB(H^2\otimes \sF)$. Recall the definition of the defect ideal $\mathcal I_\sD$ from \eqref{Defect}.
\begin{theorem}\label{T:CommIdeal}
Consider the unitarily implemented embedding
$$
\iota:C^*(V_1,V_2)\to C^*(R_q,T_z)\otimes C^*(P,U)
$$
defined by
$
\iota: A\mapsto \tau_{\operatorname{BCL}}\cdot A\cdot \tau_{\operatorname{BCL}}^*
$
for every $A\in C^*(V_1,V_2)$. Then the image of the defect ideal $\sI_{\sD}$ under $\iota$ is 
$
\sK\otimes C^*(P,U).
$
\end{theorem}

\begin{proof}
By the definition of the unitary $\tau_{\rm BCL}$, we have
$$
\iota (V_1)=T_z R_q \otimes UP + R_q\otimes UP^\perp
\quad\mbox{and}\quad
\iota (V_2)=R_{\overline{q}}T_z \otimes  P^\perp U^*+R_{\overline{q}} \otimes PU^*.
$$
For later use, we list in Table \ref{tab:three-column} the images under $\iota$ of some important commutators that lie in the defect ideal $\sI_{\sD}$.

\begin{table}[ht!]
\centering
\begin{tabular}{l c}
\toprule
$\sC$ & $\iota(\sC)$ \\
\midrule
$[V^*_1,V_1]$ & $\widetilde{P_0} \otimes UPU^*$ \\
$[V^*_2,V_2]$ & $\widetilde{P_0} \otimes P^\perp$ \\
\midrule
$V_2[V^*_1,V_1]V^*_2$ & $\widetilde{P_0} \otimes P$ \\
$V_1[V^*_2,V_2]+[V^*_1,V_1]V^*_2$ & $R_q\widetilde{P_0} \otimes U=\widetilde{P_0}\otimes U$ \\
\bottomrule
\end{tabular}
\caption{Commutators}
\label{tab:three-column}
\end{table}
The ideal $\sK\otimes C^*(P,U)$ contains the generators of $\iota(\sI_{\sD})$, so
$\iota(\sI_{\sD})\subseteq \sK\otimes C^*(P,U)$.

Since $ \widetilde{P_0}\otimes U$ and $\widetilde{P_0}\otimes P$ belong to $\iota (\sI_{\sD})$, it follows that $\widetilde{P_0}\otimes A\in \iota (\sI_{\sD})$ for any $A\in C^*(P,U)$. Moreover, because $\iota (V_1V_2)=T_z\otimes I\in \iota(C^*(V_1,V_2))$, we obtain for $n\geq 0$ that
$$
\widetilde{P_n}\otimes A = (T_z^n\otimes I)(\widetilde{P_0}\otimes A)(T_z^{*n}\otimes I)\in \iota(\sI_{\sD}).
$$
In particular, $\iota (\sI_{\sD})$ contains $F\otimes A$ for every rank-one operator $F$ and every $A\in C^*(P,U)$, and hence $\sK\otimes C^*(P,U)\subseteq \iota(\sI_{\sD})$.
\end{proof}

\noindent
Recall that for $\theta\in\mathbb R$, the rotation algebra $\mathcal A_\theta$ is the universal $C^*$-algebra generated by symbols $u$ and $v$ satisfying
$$
\{uu^*=1=u^*u,\; vv^*=1=v^*v,\; uv=e^{2\pi i \theta}vu\}.
$$
The term \emph{noncommutative torus} and \emph{irrational rotation algebra} are commonly used for $\mathcal A_\theta$ when $\theta$ is irrational. The irrational rotation algebra has been studied extensively; we refer the reader to \cite[Chapter VI]{MR1402012}. It is known that $\sA_{\theta}$ has a unique faithful trace (see \cite[Corollary VI.1.2 and Proposition VI.1.3]{MR1402012} and simple, i.e., has no nontrivial closed two-sided $*$-ideals (see, e.g., \cite[Theorem VI.1.4]{MR1402012}). Thus if $(U,V)$ is a pair of $q$-commuting unitary operators, then $C^*(U,V)$ is canonically isomorphic to $\sA_{\theta}$. For $A\in\sA_{\theta}$ define the compression $$\sigma(A)=P_{H^2}\restr{A}{H^2}:H^2\to H^2,$$ and let $\sT(\sA_\theta)$ be the $C^*$-subalgebra of $\sB(\sH^2)$ generated by $\{\sigma(A):A\in\sA_{\theta}\}$. We will use the following result of Park, which identifies $\sT(\sA_\theta)$ as an extension of $\sA_{\theta}$ by the compact operators. 

\begin{lemma}\cite[Theorem 2]{MR2176158}\label{L:SEirrational}
For irrational $\theta$, the sequence 
$$
0 \to \sK \xrightarrow{} \sT(\sA_\theta) \xrightarrow{\pi} \sA_{\theta} \to 0
$$
is short exact and admits the unital completely positive cross section $\sigma$, where $\pi$ is the canonical quotient map.
\end{lemma}

It follows from \cite[Proposition 1]{MR2176158} that $\sigma$ is an isometry, i.e., $\|\sigma(A)\|=\|A\|$ for all $A\in \sA_{\theta}$. In what follows we use the standard identifications $\sT(\sA_{\theta})\cong C^*(R_q,T_z)$ and $C^*(R_q,M_z)\cong\sA_{\theta}$ without further comment.

\begin{lemma}\label{L:SEtensored}
The sequence
$$
0 \to \sK\otimes C^*(P,U) \xrightarrow{} C^*(R_q,T_z)\otimes C^*(P,U) \xrightarrow{\pi\otimes 1} C^*(R_q,M_z)\otimes C^*(P,U) \to 0
$$
is exact.
\end{lemma}

\begin{proof}
Since the short exact sequence in Lemma~\ref{L:SEirrational} admits a completely positive cross section, Theorem \ref{T:CPcross} (applied to $C^*(P,U)$) implies that $\sK\otimes C^*(P,U)$ is a closed two-sided $*$-ideal of $C^*(R_q,T_z)\otimes C^*(P,U)$ and that the quotient is canonically $*$-isomorphic to $(C^*(R_q,T_z)/\sK)\otimes C^*(P,U)$. This gives the desired exactness.
\end{proof}
\begin{lemma}\label{L:Defect kernel}
The map $(\pi\otimes 1)\iota :C^*(V_1,V_2)\to C^*(R_q,M_z)\otimes C^*(P,U)$ is a $*$-homomorphism satisfying $\ker (\pi\otimes 1)\iota =\sI_{\sD}$ and $(\pi\otimes 1)\iota (V_j)=\widetilde{V}_j$ for $j=1,2$.
\end{lemma}

\begin{proof} It is clear that $(\pi\otimes 1)\iota$ is a $*$-homomorphism. Moreover,
$$\ker(\pi\otimes1)\iota=\iota (C^*(V_1,V_2))\cap \ker (\pi\otimes1)=\iota (C^*(V_1,V_2))\cap \bigl(\sK\otimes C^*(P,U)\bigr)=\sK\otimes C^*(P,U),$$
where last equality follows from Theorem~\ref{T:CommIdeal}.
\end{proof}
We now consider the compression map $\Psi:C^*(\widetilde{V}_1,\widetilde{V}_2)\to C^*(R_q,T_z)\otimes C^*(P,U)$ defined by $$\Psi(A)\coloneqq (P_{H^2}\otimes I)\restr{A}{H^2\otimes \sF}.$$ 

\begin{lemma}
The map $\Psi$ is linear and contractive. Moreover, $\Psi(A)\in \iota (C^*(V_1,V_2))$ for every $A\in C^*(\widetilde{V}_1,\widetilde{V}_2)$.
\end{lemma}
\begin{proof}
Linearity and contractivity of $\Psi$ are immediate. To see that $\Psi(A)\in \iota(C^*(V_1,V_2))$, it suffices to consider $*$-monomials $A=\widetilde{V_1}^{*n_1}\widetilde{V_2}^{*n_2}\widetilde{V_1}^{n_3}\widetilde{V_2}^{n_4}$, $n_i\geq 0 $. Since $\Psi(\widetilde{V_i})=\iota(V_i)$ and $\Psi(\widetilde{V_i}^*)=\iota(V_i^*)$ for $i=1,2$, we obatin
\begin{align*}\Psi(A)=(P_{H^2}\otimes I)\restr{A}{H^2\otimes \sF}=\iota (V_1^{*n_1}V_2^{*n_2}V_1^{n_3}V_2^{n_4})
\end{align*} using the properties of the unitary extension. 
\end{proof}

\begin{theorem}\label{T:Generalized Toeplitz} Define $s: C^*(\widetilde{V}_1,\widetilde{V}_2)\to C^*(V_1,V_2)$ by 
$$s(A)\coloneqq \iota^{-1}\Psi (A).$$

Then $s$ is a positive linear map and satisfies $(\pi\otimes 1)\iota s=1$ and $\|s(A)\|=\|A\|$.
\end{theorem}

\begin{proof}
The map $s$ is positive, linear and contractive because both $\Psi$ and $\iota^{-1}$ have these properties on their respective domains. To verify $(\pi\otimes1)\iota s=1$, it again suffices to consider $*$-monomials $A=\widetilde{V_1}^{*n_1}\widetilde{V_2}^{*n_2}\widetilde{V_1}^{n_3}\widetilde{V_2}^{n_4}$, $n_i\geq 0 $. Using the computation of $\Psi(A)$ above and the properties of the minimal $q$-commuting unitary extension, we get 
\begin{align*}(\pi\otimes 1)\iota s(A)=(\pi\otimes 1)\Psi(A)=A.
\end{align*}
Finally, since $(\pi\otimes1)\iota$ is contractive,
$$\|A\|=\|(\pi\otimes 1)\iota s(A)\|\leq \|s(A)\|\leq \|A\|.
$$
Thus, $s$ is an isometry.
\end{proof}

The preceding results can be summarized in the following commutative diagram:
\begin{center}
{\tiny
\begin{tikzcd}[column sep=large, row sep=large]
0 \arrow[r] 
  & \sK\otimes C^*(P,U) \arrow[r, " "] 
  & C^*(R_q,T_z)\otimes C^*(P,U) \arrow[r, "\pi\otimes 1"] 
  &  C^*(R_q,M_z)\otimes C^*(P,U) \arrow[r] 
  & 0 \\
0 \arrow[r] 
  & \sI_{\sD} \arrow[r, " "'] \arrow[u, leftrightarrow, "\iota"] 
  & C^*(V_1,V_2) \arrow[r, "(\pi\otimes 1)\iota"'] \arrow[u, "\iota"] 
  & C^*(\widetilde{V_1},\widetilde{V_2}) \arrow[r] \arrow[u, " "] 
  \arrow[ul, bend right=15, "\Psi"] 
  \arrow[l, bend right=20, "s"'] 
  & 0
\end{tikzcd}
}
\end{center}
Since the bottom row is short exact and admits the completely positive cross section $s$, the Banach space $C^*(V_1,V_2)$ splits as  $C^*(V_1,V_2)=\sI_{\sD}\oplus s(C^*(\widetilde{V}_1,\widetilde{V}_2))$. For $A\in C^*(\widetilde{V}_1,\widetilde{V}_2)$ write $T_A\coloneqq s(A)$; we view $T_A$ as a generalized Toeplitz operator with symbol $A$.
\begin{theorem}\label{T:Represent}
Every $X\in C^*(V_1,V_2)$ can be written uniquely as $X=T_A+D$ for some $A\in C^*(\widetilde{V}_1,\widetilde{V}_2)$ and $D\in\sI_{\sD}$.
\end{theorem}


\section{Fredholm theory}\label{S:FredholmTheory}
In this section we restrict our consideration to the case when the coefficient space $\sF$ is infinite dimensional as the finite dimensional case is understood; the interested reader can see the paper \cite{MR2176158} by Park.
Let
\begin{equation*}
\pi'\colon C^*(P,U,\sK)\longrightarrow C^*(P,U,\sK)/\sK
\end{equation*}
denote the Calkin quotient map.

\begin{theorem}\label{T:Fredholm}
If $\pi'$ has a completely positive linear cross section, then for $T_A+D\in C^*(V_1,V_2)$ we have $(\pi\otimes 1)\iota(T_A+D)=A$, and $T_A+D$ is Fredholm if and only if both $A$ and $(1\otimes \pi')\iota(T_A+D)$ are invertible.
In particular, the conclusion holds if $C^*(P,U,\sK)/\sK$ is commutative.
\end{theorem}

\begin{proof}
By Lemma~\ref{L:Defect kernel}, the map $(\pi\otimes 1)\iota$ is a $*$-homomorphism and satisfies $(\pi\otimes 1)\iota(D)=0$ for all $D\in \sI_{\sD}$. Moreover, $(\pi\otimes 1)\iota(T_A)=A$ for all $A\in C^*(\widetilde{V}_1,\widetilde{V}_2)$ (see Theorem~\ref{T:Generalized Toeplitz}). Hence $(\pi\otimes 1)\iota(T_A+D)=A$. The Fredholm criterion then follows from the discussion surrounding \eqref{Eq:symbol}.

As noted in the discussion preceding Theorem \ref{T:quotient}, when $C^*(P,U,\sK)/\sK$ is commutative and separable, the quotient map $\pi'$ admits a completely positive cross section.
\end{proof}

It is easy to see that $C^*(P,U,\sK)/\sK$ is commutative if and only if $PU-UP\in \sK$ i.e., $P$ and $U$ essentially commute. Let us consider the orthogonal decomposition $\sF=\ran P\oplus\ran P^{\perp}$. With respect to the decomposition, $U$ decomposes as $$U=\begin{bmatrix}PUP\mid_{\ran P} & PUP^\perp \mid_{\ran P^\perp}\\ P^\perp UP\mid_{\ran P} & P^\perp UP^\perp\mid_{\ran P^\perp}
\end{bmatrix}.
$$
A direct computation shows $$PU-UP=\begin{bmatrix}0 & PUP^\perp \mid_{\ran P^\perp}\\ -P^\perp UP\mid_{\ran P} & 0
\end{bmatrix}.$$
Hence $PU-UP$ is compact if and only if  $PUP^\perp $ and $ P^\perp UP $ are both compact.

In the case when $C^*(P,U,\sK)/\sK$ is commutative we obtain a concrete index theory for Fredholm operators in $C^*(V_1,V_2)$ of the form $I+D$, where $D$ belongs to the defect ideal $\sI_{\sD}$. We will use this fact frequently to construct a unitary invariant in the next section.
Let $M$ be the maximal ideal space of $C^*(\pi'U,\pi'P)$. Consider the homeomorphic embedding $\Phi\colon M\to \C^2$ defined by
$$
 \Phi(h)=\bigl(h(\pi'U),h(\pi'P)\bigr), \quad h\in M.
$$
Put $Y\coloneqq \Phi(M)$. Then $Y\subseteq \T\times\{0,1\}$, and by Gelfand duality
$$
C^*(\pi'U,\pi'P)\cong C(Y),
$$
where the coordinate functions correspond to $\pi'U\mapsto w$ and $\pi'P\mapsto \delta$.
\begin{definition}\label{D:maximal}
We say that $Y$ is \emph{maximal} if for every $w\in\T$ both $(w,0)$ and $(w,1)$ lie in $Y$, i.e., when $Y$ is the disjoint union of two copies of the unit circle.
\end{definition}
It follows from Lemma~\ref{T:CommIdeal} that $(1\otimes \pi')\iota$ maps $\sI_{\sD}$ onto $C\bigl(Y,\sK(H^2)\bigr)$.
Let $\widetilde{\sK(H^2)}$ be the group of all invertible operators on $H^2$ of the form $I+K$ with $K\in \sK(H^2)$.
By Theorem~\ref{T:Fredholm}, $I+D$ is Fredholm if and only if
$$
(1\otimes \pi')\iota (I+D)\in C\bigl(Y,\widetilde{\sK(H^2)}\bigr).
$$
Since $(\pi\otimes 1)\iota(I+D)=I$, the Fredholm index depends only on the homotopy class of the map $(1\otimes \pi')\iota (I+D)\colon Y\to \widetilde{\sK(H^2)}$.

We briefly recall the relevant homotopy computation from \cite[Section 4]{MR467392}.
The set $[Y,\widetilde{\sK(H^2)}]$ consisting of homotopy equivalent classes of maps from $Y$ into the group $\widetilde{\sK(H^2)}$, is an abelian group with respect to pointwise multiplication.
If $Y$ is maximal, then
$[Y,\widetilde{\sK(H^2)}]=\langle [f_1],[f_2]\rangle$, where
$$
 f_1(y)=I+(1-\delta)(w-1)\widetilde{P_0},
 \qquad
 f_2(y)=I+\delta(\overline{w}-1)\widetilde{P_0},
$$
with corresponding Fredholm operators
\begin{align}\label{B1B2}
B_1=I+\widetilde{P_0}\otimes P^{\perp}(U-I),
\qquad
B_2=I+\widetilde{P_0}\otimes P(U^*-I),
\end{align}
respectively, and a direct computation shows that $\ind (B_1)=\ind (B_2)$. Here recall that $\widetilde{P_0}$ is the orthogonal projection onto the space of constant functions in $H^2$. Thus if $B=I+D$ is Fredholm with $D\in \sI_{\sD}$ and $(1\otimes \pi')\iota (B)$ is homotopic to $ f_1^{n_1}f_2^{n_2}$ for some integers $n_1,n_2$, then
\begin{equation}\label{Eq:ind(B_1)}
  \ind (B)=\ind (B_1)(n_1+n_2).  
\end{equation}

We end this section with a result to be used in what follows.
\begin{theorem}\label{T:semifredimplesfred}
Suppose $(V_1,V_2)$ is a $q$-commuting isometric pair with $(\sF,P,U)$ as its BCL triple such that $C^*(P,U,\sK)/\sK$ is commutative. If $\alpha I+D$ is semi-Fredholm for some $D\in \sI_{\sD}$ and $\alpha\in \C\setminus \{0\}$, then $\alpha I+D$ is Fredholm.
\end{theorem}

\begin{proof}
As discussed earlier, $\alpha I+D$ is semi-Fredholm if and only if $$
\alpha I+(1\otimes \pi')\iota(D)(y)=g(y)$$ is left or right invertible in $\{\mu I+C(Y,K):\mu\in \C\}$ for each $y\in Y$. An application of the classical Fredholm alternative (see e.g., \cite[Theorem 5.22]{MR1634900}) yields the invertibility of $g(y)$ for each $y$. Hence $g$ is invertible, and the conclusion follows. 
\end{proof}

\section{Unitary Invariant}\label{S:Unitary invariant}
In this section we impose an additional hypothesis that guarantees the existence of a completely positive cross section of the quotient map $\pi'$, so that the Fredholm criterion in Theorem~\ref{T:Fredholm} applies.  Assume that $(V_1,V_2)$ is essentially doubly $q$-commuting with $(\sF,P,U)$ as the associated BCL triple. A direct computation gives
$$
V_2^*V_1-qV_1V_2^*
=R_q\widetilde{P_0}R_q\otimes UPUP^{\perp}
=\widetilde{P_0}\otimes UPUP^{\perp}.
$$
Thus $(V_1,V_2)$ is essentially doubly $q$-commuting if and only if $PUP^{\perp}$ is compact, in which case
$$
(PUP)(PU^*P) = P + \operatorname{compact}.
$$ Thus the restriction $PUP\vert_{\mathrm{ran}P}$ is either Fredholm or semi-Fredholm of index $+\infty$. Moreover,
\begin{lemma}\label{L:FredQuotient}
If $(V_1,V_2)$ is essentially doubly $q$-commuting with $(\sF,P,U)$ as the associated BCL triple, then $PUP\vert_{\mathrm{ran}P}$ is Fredholm if and only if $C^*(P,U,\sK)/\sK$ is commutative.
\end{lemma}
\begin{proof}
    If $C^*(P,U,\sK)/\sK$ is commutative, then and only then $PU-UP$ is compact. This implies $P^\perp UP$ is compact. Now apply Lemma~\ref{L:block-unitary-index} to the unitary $U$ with respect to the decomposition $\sF = \operatorname{ran}P \oplus \operatorname{ran}P^\perp$ to get $PUP|_{\operatorname{ran}P}$ Fredholm. Conversely, if $PUP|_{\operatorname{ran}P}$ is Fredholm, then by Lemma \ref{L:block-unitary-index} again, both $P^\perp UP|_{\operatorname{ran}P}$ and $P UP^\perp|_{\operatorname{ran}P^\perp}$ are compact, which is same as $PU-UP$ being compact, or equivalently, the quotient $C^*(P,U,\sK)/\sK$ is commutative. 
\end{proof}
We now define
$$
\kappa(V_1,V_2)\coloneqq \ind \bigl(PUP\vert_{\mathrm{ran}P}\bigr)\in \Z\cup\{+\infty\}.
$$

\begin{theorem}\label{T:kappa_pair_invariant}
For an essentially doubly $q$-commuting isometric pair $(V_1,V_2)$, the quantity $\kappa(V_1,V_2)$ is a unitary invariant of the pair.
\end{theorem}

\begin{proof}
Let $(V_1,V_2)$ and $(V_1',V_2')$ be essentially doubly $q$-commuting pairs of isometries on $\sH$ and $\sH'$, respectively, and let $(\sF,P,U)$ and $(\sF',P',U')$ be their corresponding BCL triples. By \Cref{T:q-BCL}, the pairs $(V_1,V_2)$ and $(V_1',V_2')$ are unitarily equivalent if and only if the triples $(\sF,P,U)$ and $(\sF',P',U')$ are unitarily equivalent. This implies $PUP\vert_{\mathrm{ran}P}$ and $P'U'P'\vert_{\mathrm{ran}P'}$ are unitarily equivalent, and hence have the same Fredholm index.
\end{proof}
As observed in the discussion preceding Theorem \ref{T:quotient}, when the quotient $C^*(P,U,\sK)/\sK$ is commutative, the quotient map $\pi'$ admits a completely positive cross section. 
 
The result below shows that $\pi'$ admits a completely positive cross section even when $(V_1,V_2)$ is essentially doubly $q$-commuting.

\begin{theorem}\label{T:cp_section}
For an essentially doubly $q$-commuting isometric pair $(V_1,V_2)$, the quotient map $\pi'$ admits a completely positive cross section. Moreover, if $\kappa(V_1,V_2)=0$ then $\pi'$ admits a $*$-isomorphic cross section.
\end{theorem}
\begin{proof}
Let $(V_1,V_2)$ be essentially doubly $q$-commuting with BCL triple $(\sF,P,U)$. Then the associated commuting pair $(R_{\overline{q}}V_1, V_2R_{q})$ is essentially doubly commuting with the same BCL triple. By \cite[Theorem 5.3]{MR467392}, the quotient map $\pi'$ admits a completely positive cross section. If, in addition, $\kappa(V_1,V_2)=0$, then \cite[Lemma 7.1]{MR467392} shows that $\pi'$ admits a $*$-isomorphic cross section.
\end{proof}
\begin{theorem}\label{indformula}
Assume that $\kappa(V_1,V_2)$ is finite and nonzero. Let $Y$, as in the discussion preceding Theorem \ref{T:semifredimplesfred}, be the homeomorphic image of the maximal ideal space of the (commutative) quotient algebra $C^*(P,U,\sK)/\sK$ associated to $(V_1,V_2)$. Then $Y$ is maximal in the sense of Definition \ref{D:maximal}, and every Fredholm operator of the form $I+D$ with $D\in\sI_{\sD}$ satisfies
$$
\ind (I+D)=-\kappa(V_1,V_2)(n_1+n_2),
$$
where $(1\otimes \pi')\iota(I+D)$ is homotopic to $f_1^{n_1}f_2^{n_2}$ for some integers $n_1,n_2$.
\end{theorem} 
\begin{proof} The proof that $Y$ is maximal is the same as in \cite[Theorem 6.1]{MR467392}. By the index formula \eqref{Eq:ind(B_1)}, it suffices to show that $\ind(B_1)=-\kappa(V_1,V_2)$, where $B_1$ is the Fredholm operator as in \eqref{B1B2}. With respect to the decomposition $H^2(\sF)=\sF\oplus (\operatorname{ran}\widetilde{P_0}^\perp\otimes \sF)$, the operator $B_1$ is the direct sum of $P^\perp U-P$ acting on $\sF$ and the identity operator on $\operatorname{ran}\widetilde{P_0}^\perp\otimes \sF$. Thus $\ind(B_1)=\ind(P^\perp U-P)$. With respect to the decomposition $\sF = \operatorname{ran}P +\operatorname{ran}P^\perp$, we have
\begin{align}\label{Computation}
P^\perp U - P = \begin{bmatrix}
    -I_{\operatorname{ran}P} & 0\\
    0 & P^\perp UP^\perp|_{\operatorname{ran}P^\perp}
\end{bmatrix}  + \begin{bmatrix}
     0 & 0\\
     P^\perp UP|_{\operatorname{ran}P} & 0
\end{bmatrix}.
\end{align}Now $P^\perp UP = UP - PUP = UP - PU + PUP^\perp$. Since $(V_1,V_2)$ is essentially doubly $q$-commuting, $PUP^\perp$ is compact. Since $\kappa(V_1,V_2)$ is finite, by Lemma \ref{L:FredQuotient}, $PU-UP$ is compact, hence $P^\perp UP$ is compact. Since the index is invariant under compact perturbation, by \eqref{Computation}, we have
$$
\ind(B_1) = \ind(P^\perp U -P) = \ind(P^\perp U P^\perp|_{\operatorname{ran}P^\perp} = -\ind(PUP|_{\operatorname{ran}P}) = -\kappa(V_1,V_2),
$$where to obtain the second equality we used Lemma \ref{L:block-unitary-index}.
\end{proof}

We now consider a distinguished element of the form $I+D$ with $D$ belongs to the defect ideal, namely $I+D(V_1,V_2)$, where
\[
D(V_1,V_2)\coloneqq [V_2^*,V_2](V_1-I)[V_2^*,V_2].
\]
Clearly, $D(V_1,V_2)\in\sI_{\sD}$. The next theorem shows that $I+D(V_1,V_2)$ is semi-Fredholm and provides an explicit index formula.

\begin{theorem}\label{T:I+D}
For an essentially doubly $q$-commuting isometric pair $(V_1,V_2)$ with infinite dimensional coefficient space, the operator $I+D(V_1,V_2)$ is semi-Fredholm and
$$
\kappa(V_1,V_2)=-\ind (I+D(V_1,V_2))
$$

\end{theorem}
\begin{proof}
A direct computation gives $I+D(V_1,V_2)=V_2V_2^*+[V_2^*,V_2]V_1[V_2^*,V_2]$. By Theorem \ref{T:CommIdeal}, this operator unitarily equivalent to
\begin{align*}
&I-(\widetilde{P_0}\otimes P^\perp)+ (\widetilde{P_0}\otimes P^\perp)(T_z R_q \otimes UP + R_q\otimes UP^\perp)(\widetilde{P_0}\otimes P^\perp)\\
&= (I-\widetilde{P_0})\otimes I+\widetilde{P_0}\otimes P+ \widetilde{P_0}R_q\widetilde{P_0}\otimes P^\perp UP^\perp=(I-\widetilde{P_0})\otimes I+\widetilde{P_0}\otimes (P+P^\perp UP^\perp),
\end{align*}
where we used $\widetilde{P_0}R_q\widetilde{P_0}=\widetilde{P_0}$. With respect to the decomposition $H^2\otimes \sF=\widetilde{P_0}H^2\otimes\sF\oplus(I-\widetilde{P_0})H^2\otimes\sF=\C e_0\otimes\sF\oplus(I-\widetilde{P_0})H^2\otimes\sF$, we obtain the block diagonal form

$$\iota(I+D(V_1,V_2))=\begin{bmatrix}
P+P^\perp UP^\perp & 0\\
0&I  
\end{bmatrix}.$$
Thus it suffices to show that $P+P^\perp UP^\perp$ is semi-Fredholm on $\sF$. Equivalently, the restriction $P^\perp UP^\perp\vert_{\ran P^\perp}$ is semi-Fredholm on $\ran P^\perp$ and $$\ind (P^\perp UP^\perp\mid_{\ran P^\perp})=-\ind PUP\mid_{\textit{ran} P}.$$ Since $PUP^\perp$ is compact, we apply Lemma~\ref{L:block-unitary-index} to the block decomposition of the unitary $$U=\begin{bmatrix}PUP\vert_{\ran P} & PUP^\perp \vert_{\ran P^\perp}\\ P^\perp UP\vert_{\ran P} & P^\perp UP^\perp\vert_{\ran P^\perp}
\end{bmatrix}
$$ and obtain the desired conclusion.
\end{proof}

\begin{theorem}\label{T:ind(I+D)=0}
If $\kappa(V_1,V_2)=0$ and $D$ in the defect ideal $\sI_{\sD}$ is such that $I+D$ is Fredholm, then $\ind(I+D)=0$.
\end{theorem}

\begin{proof}By Theorem \ref{T:cp_section}, if $\kappa(V_1,V_2)=0$ then the quotient map $\pi'$ admits a $*$-isomorphic cross section, say $\rho$. By Theorem \ref{T:CommIdeal}, we also have $\iota(\sI_{\sD})=\sK\otimes C^*(P,U)$. Suppose that $I+D$ is Fredholm. Then $I+(1\otimes\pi')\iota(D)$ is invertible, and hence so is $I+(1\otimes\rho\pi')\iota(D)$. Since $\iota(D)\in\sK\otimes C^*(P,U)$, the difference $(1\otimes\rho\pi')\iota(D)-\iota(D)$ is compact. Therefore $\iota(I+D)$ is a compact perturbation of an invertible operator, and hence is Fredholm of index $0$. Since $\iota$ is unitarily implemented, the same holds for $I+D$.
\end{proof}

\begin{lemma}\label{L:ideal_equality}
Let $(V_1,V_2)$ and $(V_1',V_2')$ be $q$-commuting isometries on $\sH$ such that $C^*(V_1,V_2)=C^*(V_1',V_2')$. Then $\sI_{\sD}=\sI_{\sD'}$.
\end{lemma}

\begin{proof}
By symmetry it suffices to show that $\sI_{\sD'}\subseteq \sI_{\sD}$. It is enough to prove that $I-V_i'V_i'^*\in \sI_{\sD}$ for $i=1,2$.
Let $\pi\colon C^*(V_1,V_2)\to C^*(V_1,V_2)/\sI_{\sD}$ be the quotient map and set
$B\coloneqq C^*(\pi(V_1),\pi(V_2))$. Since $V_i'\in C^*(V_1,V_2)$ is an isometry, $\pi(V_i')$ is an isometry in $B$. Once we show that $\pi(V_i')$ is unitary in $B$, we obtain $\pi(I-V_i'V_i'^*)=0$, and hence $I-V_i'V_i'^*\in\ker\pi=\sI_{\sD}$.

Recall the discussion preceding Lemma \ref{L:SEirrational}. Since $(\pi(V_1),\pi(V_2))$ is a $q$-commuting unitary pair, the algebra $B$ is canonically isomorphic to the irrational rotation algebra $\sA_{\theta}$. Moreover, since $\sA_\theta$ admits a (unique) faithful trace, every isometry in $\sA_{\theta}$ is unitary. Indeed, if $v\in\sA_{\theta}$ is an isometry and $\tau$ denotes the faithful trace, then  $$\tau(1-vv^*)=\tau(1)-\tau(vv^*)=\tau(1)-\tau(v^*v)=0.$$ By faithfulness of $\tau$, we conclude $vv^*=1$.
\end{proof}

We now have all the necessary tools to prove the main theorem of the paper.
\begin{theorem}\label{T:kappa_Cstar_invariant}
Let $(V_1,V_2)$ be a pair of $q$-commuting isometries such that $V_2^*V_1-qV_1V_2^*$ is compact. Then $|\kappa(V_1,V_2)|$ is a unitary invariant for $C^*(V_1,V_2)$.
\end{theorem}

\begin{proof}

Let $\tau\colon \sH\to \sH'$ be a unitary such that $\tau C^*(V_1,V_2)\tau^*=C^*(V_1',V_2')$ and $(V_1'',V_2'')\coloneqq \tau^*(V_1',V_2')\tau$.
Then $(V_1'',V_2'')$ is essentially doubly $q$-commuting and $C^*(V_1,V_2)=C^*(V_1'',V_2'').$
Since $(V_1'',V_2'')$ is unitarily equivalent to $(V_1',V_2')$, Theorem~\ref{T:kappa_pair_invariant} gives $\kappa(V_1',V_2')=\kappa(V_1'',V_2'')$. Hence it suffices to prove that $|\kappa(V_1,V_2)|=|\kappa(V_1'',V_2'')|$.
We consider three cases, according to the value of $\kappa(V_1,V_2)$.

\textbf{Case I:} Assume $\kappa(V_1,V_2)=0$. By \Cref{T:I+D} applied to $(V_1'',V_2'')$, the operator $I+D(V_1'',V_2'')$ is semi-Fredholm of index $-\kappa(V_1'',V_2'')$. Since $\sI_{\sD}=\sI_{\sD''}$, we may view  $I+D(V_1'',V_2'')$ as an element of $ C^*(V_1,V_2)$ of the form $I+D$ with $D\in \sI_{\sD}$. Since $(V_1,V_2)$ is essentially doubly $q$-commuting, by Lemma \ref{L:FredQuotient} and the discussion preceding it,  $\kappa(V_1,V_2)=0$ implies $C^*(P,U,\sK)/\sK$ is commutative. By \Cref{T:semifredimplesfred}, $I+D(V_1'',V_2'')=I+D$ is in fact Fredholm. Then by \Cref{T:ind(I+D)=0}, 
$$
0=\ind(I+D)=\ind (I+D(V_1'',V_2''))=\kappa (V_1'',V_2'')
$$ and hence $\kappa(V_1'',V_2'')=0$ as was required to show.

Let us note in addition that the argument above can be reversed to show $\kappa(V_1'',V_2'')=0$ implies $\kappa(V_1,V_2)=0$. For this direction, one needs to work with the semi-Fredholm element $I+D(V_1,V_2)$ instead follow the same line of arguments as given above.

\textbf{Case II:} Assume $0<|\kappa(V_1,V_2)|<\infty$. As observed in Case~I, we may rule out the possibility $\kappa(V_1'',V_2'')=0$. Since $\sI_\sD=\sI_{\sD''}$, we can view $I+D(V_1'',V_2'')$ as an element of $C^*(V_1,V_2)$. By the assumption that $\kappa(V_1,V_2)<\infty$, the associated quotient $C^*(P,U,\sK)/\sK$ is commutative (by Lemma \ref{L:FredQuotient}). This makes Theorem \ref{T:semifredimplesfred} applicable to the pair $(V_1,V_2)$, which implies that $I+D(V_1'',V_2'')$, viewed as an element of $C^*(V_1,V_2)$, is Fredholm. By \Cref{T:I+D} applied to $(V_1'',V_2'')$, 
$$
\kappa(V_1'',V_2'') =  - \ind(I+D(V_1'',V_2'')) <\infty.
$$

Thus $\kappa(V_1'',V_2'')$ can neither be zero nor infinite. It remains to show that $|\kappa(V_1,V_2)|=|\kappa(V_1'',V_2'')|$. Since $k(V_1'',V_2'')$ is nonzero and finite, we apply Theorem \ref{indformula} for the pair $(V_1'',V_2'')$ and the element $D(V_1,V_2)\in \sI_\sD=\sI_{\sD''}$ to get an integer $N''$ such that
\begin{align}\label{Aux'}
N''\cdot \kappa(V_1'',V_2'') = \ind(I+D(V_1,V_2))=\kappa(V_1,V_2),
\end{align}where to obtain the last equality we applied Theorem \ref{T:I+D} to $(V_1,V_2)$. Since $k(V_1,V_2)$ is also nonzero and finite, Theorem \ref{indformula} can also be applied to $(V_1,V_2)$ for the element $D(V_1'',V_2'')\in \sI_{\sD''}=\sI_\sD$ to get an integer $N$ so that
\begin{align}\label{Aux''}
N\cdot \kappa(V_1,V_2) = \ind(I+D(V_1'',V_2''))=\kappa(V_1'',V_2''),
\end{align}where, as before, to obtain the last equality we applied Theorem \ref{T:I+D} to $(V_1'',V_2'')$. Equations \eqref{Aux'} and \eqref{Aux''} imply that $\kappa(V_1,V_2)$ and $\kappa(V_1'',V_2'')$ divides each other, hence $|\kappa(V_1,V_2)|=|\kappa(V_1'',V_2'')|$.

Finally, the set of arguments above can be reversed to obtain the reverse implication: 
$$
0<|\kappa(V_1'',V_2'')|<\infty\:\Longrightarrow 0<|\kappa(V_1,V_2)|<\infty.
$$

\textbf{Case III:} Assume $\kappa(V_1,V_2)=\infty$. The possibilities $\kappa(V_1'',V_2'')=0$ and $0<|\kappa(V_1'',V_2'')|<\infty$ are ruled out by the previous arguments in Cases I and II. Hence $\kappa(V_1'',V_2'')=\infty$ as well.
\end{proof}

\section*{Acknowledgments}
Part of this work was completed when the first author visited the Indian Statistical Institute, Bangalore. He gratefully acknowledges the support and hospitality extended by the institution.

\bibliographystyle{plain}
\bibliography{ref}

@article {MR3022732,
    AUTHOR = {Weber, M.},
     TITLE = {On {$C^\ast$}-algebras generated by isometries with twisted
              commutation relations},
   JOURNAL = {J. Funct. Anal.},
  FJOURNAL = {Journal of Functional Analysis},
    VOLUME = {264},
      YEAR = {2013},
    NUMBER = {8},
     PAGES = {1975--2004},
      ISSN = {0022-1236,1096-0783},
   MRCLASS = {46L05 (46L54 46L80)},
  MRNUMBER = {3022732},
MRREVIEWER = {William\ Paschke},
       DOI = {10.1016/j.jfa.2013.02.001},
       URL = {https://doi.org/10.1016/j.jfa.2013.02.001},
}

@article {MR2176158,
    AUTHOR = {Park, E.},
     TITLE = {Toeplitz algebras and extensions of irrational rotation
              algebras},
   JOURNAL = {Canad. Math. Bull.},
  FJOURNAL = {Canadian Mathematical Bulletin. Bulletin Canadien de
              Math\'ematiques},
    VOLUME = {48},
      YEAR = {2005},
    NUMBER = {4},
     PAGES = {607--613},
      ISSN = {0008-4395,1496-4287},
   MRCLASS = {46L80 (47A53 47B35)},
  MRNUMBER = {2176158},
MRREVIEWER = {Sriwulan\ Adji},
       DOI = {10.4153/CMB-2005-056-2},
       URL = {https://doi.org/10.4153/CMB-2005-056-2},
}

@article {MR2320853,
    AUTHOR = {Jury, M. T.},
     TITLE = {The {F}redholm index for elements of {T}oeplitz-composition
              {$C^*$}-algebras},
   JOURNAL = {Integral Equations Operator Theory},
  FJOURNAL = {Integral Equations and Operator Theory},
    VOLUME = {58},
      YEAR = {2007},
    NUMBER = {3},
     PAGES = {341--362},
      ISSN = {0378-620X,1420-8989},
   MRCLASS = {47B33 (46L55 47A53 47B35 47L80)},
  MRNUMBER = {2320853},
MRREVIEWER = {A.\ B\"ottcher},
       DOI = {10.1007/s00020-007-1494-0},
       URL = {https://doi.org/10.1007/s00020-007-1494-0},
}

@article {MR467392,
    AUTHOR = {Berger, C. A. and Coburn, L. A. and Lebow, A.},
     TITLE = {Representation and index theory for {$C\sp*$}-algebras
              generated by commuting isometries},
   JOURNAL = {J. Functional Analysis},
  FJOURNAL = {Journal of Functional Analysis},
    VOLUME = {27},
      YEAR = {1978},
    NUMBER = {1},
     PAGES = {51--99},
      ISSN = {0022-1236},
   MRCLASS = {47C10 (46L10 46M20)},
  MRNUMBER = {467392},
MRREVIEWER = {Horst\ Behncke},
       DOI = {10.1016/0022-1236(78)90019-8},
       URL = {https://doi.org/10.1016/0022-1236(78)90019-8},
}

@incollection {MR5075249,
    AUTHOR = {Ball, J. A. and Sau,H.},
     TITLE = {Models for {$q$}-commuting and doubly {$q$}-commuting pairs of
              isometries},
 BOOKTITLE = {Schur analysis and applications to hypercomplex analysis,
              neural networks, and linear systems},
    SERIES = {Oper. Theory Adv. Appl.},
    VOLUME = {308},
     PAGES = {27--72},
 PUBLISHER = {Birkh\"auser/Springer, Cham},
      YEAR = {[2026] \copyright 2026},
      ISBN = {978-3-032-02314-8; 978-3-032-02315-5},
   MRCLASS = {47A65 (47A13 47A45)},
  MRNUMBER = {5075249},
       DOI = {10.1007/978-3-032-02315-5\_2},
       URL = {https://doi.org/10.1007/978-3-032-02315-5_2},
}

@book {MR1634900,
    AUTHOR = {Douglas, R. G.},
     TITLE = {Banach algebra techniques in operator theory},
    SERIES = {Graduate Texts in Mathematics},
    VOLUME = {179},
   EDITION = {Second},
 PUBLISHER = {Springer-Verlag, New York},
      YEAR = {1998},
     PAGES = {xvi+194},
      ISBN = {0-387-98377-5},
   MRCLASS = {47-01 (46-01)},
  MRNUMBER = {1634900},
       DOI = {10.1007/978-1-4612-1656-8},
       URL = {https://doi.org/10.1007/978-1-4612-1656-8},
}

@article {MR2132080,
    AUTHOR = {J{\o}rgensen, P. E. T. and Proskurin, D. P. and Samo{\u\i}lenko, Y. S.},
     TITLE = {On {$C^*$}-algebras generated by pairs of {$q$}-commuting
              isometries},
   JOURNAL = {J. Phys. A},
  FJOURNAL = {Journal of Physics. A. Mathematical and General},
    VOLUME = {38},
      YEAR = {2005},
    NUMBER = {12},
     PAGES = {2669--2680},
      ISSN = {0305-4470,1751-8121},
   MRCLASS = {46L05 (81R15 81S05)},
  MRNUMBER = {2132080},
MRREVIEWER = {Lyudmila\ Turowska},
       DOI = {10.1088/0305-4470/38/12/009},
       URL = {https://doi.org/10.1088/0305-4470/38/12/009},
}

@article {MR2091468,
    AUTHOR = {Ga\c{s}par, D. and Ga\c{s}par, P.},
     TITLE = {Wold decompositions and the unitary model for bi-isometries},
   JOURNAL = {Integral Equations Operator Theory},
  FJOURNAL = {Integral Equations and Operator Theory},
    VOLUME = {49},
      YEAR = {2004},
    NUMBER = {4},
     PAGES = {419--433},
      ISSN = {0378-620X,1420-8989},
   MRCLASS = {47A45 (42B10 47A13)},
  MRNUMBER = {2091468},
MRREVIEWER = {Alexey\ S.\ Tikhonov},
       DOI = {10.1007/s00020-002-1216-6},
       URL = {https://doi.org/10.1007/s00020-002-1216-6},
}

@article {MR229093,
    AUTHOR = {Suciu, I.},
     TITLE = {On the semi-groups of isometries},
   JOURNAL = {Studia Math.},
  FJOURNAL = {Polska Akademia Nauk. Instytut Matematyczny. Studia
              Mathematica},
    VOLUME = {30},
      YEAR = {1968},
     PAGES = {101--110},
      ISSN = {0039-3223,1730-6337},
   MRCLASS = {47.50},
  MRNUMBER = {229093},
MRREVIEWER = {Ronald\ G.\ Douglas},
       DOI = {10.4064/sm-30-1-101-110},
       URL = {https://doi.org/10.4064/sm-30-1-101-110},
}

@article {MR2289758,
    AUTHOR = {Bercovici, H. and Douglas, R. G. and Foias, C.},
     TITLE = {On the classification of multi-isometries},
   JOURNAL = {Acta Sci. Math. (Szeged)},
  FJOURNAL = {Acta Universitatis Szegediensis. Acta Scientiarum
              Mathematicarum},
    VOLUME = {72},
      YEAR = {2006},
    NUMBER = {3-4},
     PAGES = {639--661},
      ISSN = {0001-6969,2064-8316},
   MRCLASS = {47A13 (47A45)},
  MRNUMBER = {2289758},
MRREVIEWER = {Gelu\ Fanica\ Popescu},
}

@incollection {MR4978050,
    AUTHOR = {Ball, J. A. and Sau, H.},
     TITLE = {Commuting and doubly commuting pairs of isometries},
 BOOKTITLE = {Operator theory, related fields, and applications},
    SERIES = {Oper. Theory Adv. Appl.},
    VOLUME = {307},
     PAGES = {15--67},
 PUBLISHER = {Birkh\"auser/Springer, Cham},
      YEAR = {[2025] \copyright 2025},
      ISBN = {978-3-032-00154-2; 978-3-032-00155-9},
   MRCLASS = {47-02 (47A15 47A45)},
  MRNUMBER = {4978050},
       DOI = {10.1007/978-3-032-00155-9\_2},
       URL = {https://doi.org/10.1007/978-3-032-00155-9_2},
}

@book {MR5090861,
    AUTHOR = {Ball, J. A. and Sau, H.},
     TITLE = {Dilation and model theory for pairs of commuting contraction
              operators},
    SERIES = {Cambridge Tracts in Mathematics},
    VOLUME = {236},
 PUBLISHER = {Cambridge University Press, Cambridge},
      YEAR = {2026},
     PAGES = {xii+302},
      ISBN = {978-1-009-68721-8},
   MRCLASS = {47-02 (47A13 47A20 47A45 47A48 47A56)},
  MRNUMBER = {5090861},
}

@article {MR69403,
    AUTHOR = {Stinespring, W. F.},
     TITLE = {Positive functions on {$C^*$}-algebras},
   JOURNAL = {Proc. Amer. Math. Soc.},
  FJOURNAL = {Proceedings of the American Mathematical Society},
    VOLUME = {6},
      YEAR = {1955},
     PAGES = {211--216},
      ISSN = {0002-9939,1088-6826},
   MRCLASS = {46.0X},
  MRNUMBER = {69403},
MRREVIEWER = {J.\ Dixmier},
       DOI = {10.2307/2032342},
       URL = {https://doi.org/10.2307/2032342},
}

@article {MR355560,
    AUTHOR = {Andersen, T. B.},
     TITLE = {Linear extensions, projections, and split faces},
   JOURNAL = {J. Functional Analysis},
  FJOURNAL = {Journal of Functional Analysis},
    VOLUME = {17},
      YEAR = {1974},
     PAGES = {161--173},
      ISSN = {0022-1236},
   MRCLASS = {46E15 (46A99 46L05)},
  MRNUMBER = {355560},
MRREVIEWER = {T.\ Ando},
       DOI = {10.1016/0022-1236(74)90010-x},
       URL = {https://doi.org/10.1016/0022-1236(74)90010-x},
}

@incollection {MR2931928,
    AUTHOR = {Bercovici, H. and Douglas, R.  G. and Foias, C.},
     TITLE = {Canonical models for bi-isometries},
 BOOKTITLE = {A panorama of modern operator theory and related topics},
    SERIES = {Oper. Theory Adv. Appl.},
    VOLUME = {218},
     PAGES = {177--205},
 PUBLISHER = {Birkh\"auser/Springer Basel AG, Basel},
      YEAR = {2012},
      ISBN = {978-3-0348-0220-8},
   MRCLASS = {47A45 (47A15 47B37)},
  MRNUMBER = {2931928},
MRREVIEWER = {Edward\ Azoff},
       DOI = {10.1007/978-3-0348-0221-5\_7},
       URL = {https://doi.org/10.1007/978-3-0348-0221-5_7},
}

@book {MR1402012,
    AUTHOR = {Davidson, K. R.},
     TITLE = {{$C^*$}-algebras by example},
    SERIES = {Fields Institute Monographs},
    VOLUME = {6},
 PUBLISHER = {American Mathematical Society, Providence, RI},
      YEAR = {1996},
     PAGES = {xiv+309},
      ISBN = {0-8218-0599-1},
   MRCLASS = {46Lxx (46-01)},
  MRNUMBER = {1402012},
MRREVIEWER = {Robert\ S.\ Doran},
       DOI = {10.1090/fim/006},
       URL = {https://doi.org/10.1090/fim/006},
}
\end{document}